\documentclass{siamart1116}

\usepackage{makecell,colortbl}

\usepackage[T1]{fontenc}      
\usepackage[utf8]{inputenc}   
\usepackage{microtype}        
\usepackage{bbm}              

\usepackage{lipsum}
\usepackage{amsfonts}
\usepackage{graphicx}
\usepackage{epstopdf}
\usepackage{algorithmic}

\usepackage{amsmath}
\usepackage{amssymb}
\usepackage{commath}          

\usepackage{mathtools}
\usepackage{thmtools}

\Crefname{subsection}{Section}{Sections}

\usepackage{enumerate}            
\usepackage{verbatim}             
\usepackage{etoolbox}             
\usepackage[perpage]{footmisc}    

\numberwithin{equation}{section}

\numberwithin{theorem}{section} 
\numberwithin{table}{section}
\numberwithin{figure}{section}
\declaretheorem[name=Example,style=definition,numberlike=theorem,Refname={Example,Examples}]{example}
\declaretheorem[name=Remark,style=definition,numberlike=theorem,Refname={Remark,Remarks}]{remark}

\declaretheorem[name=Assumption,style=definition,numberlike=theorem,Refname={Assumption,Assumptions}]{assumption} 

\allowdisplaybreaks

\DeclareUnicodeCharacter{00A0}{ } 

\DeclarePairedDelimiterX\Set[2]{\lbrace}{\rbrace}%
{ #1 \,:\, #2 } 
\DeclarePairedDelimiterX\inprod[2]{\langle}{\rangle}%
{ #1 , #2 } 

\newcommand{\bigO}{\mathcal{O}} 

\DeclareMathOperator*{\argmin}{arg\,min}  
\DeclareMathOperator*{\argmax}{arg\,max}  

\DeclareMathOperator\erf{erf} 

\newcommand{\expec}{\mathbb{E}}           
\newcommand{\gauss}{\mathcal{N}}          

\renewcommand{\epsilon}{\varepsilon}
\renewcommand{\phi}{\varphi}

\newcommand{\R}{\mathbb{R}} 
\newcommand{\N}{\mathbb{N}} 

\DeclareMathOperator{\neper}{e} 

\newcommand{\transpose}{\mathsf{T}}         

\newcommand{\vct}[1]{\mathbf{#1}} 
\newcommand{\vctg}{\pmb}          

\newcommand{\mK}{\vct{K}}

\newcommand{\mS}{\vct{S}}

\newcommand{\mP}{\vct{P}}

\newcommand{\mZero}{\vct{0}}

\newcommand{\ma}{\vct{a}}
\newcommand{\mb}{\vct{b}}
\newcommand{\mx}{\vct{x}}
\newcommand{\my}{\vct{y}}

\newcommand{\mm}{\vct{m}}
\newcommand{\mq}{\vct{q}}

\newcommand{\mz}{\vct{z}}
\newcommand{\muu}{\vct{u}}

\newcommand{\ms}{\vct{s}}
\newcommand{\mw}{\vct{w}}

\newcommand{\mlambda}{\vctg\lambda}

\usepackage{tikz}
\usepackage{pgfplots}

\usepackage[font=footnotesize,labelfont=bf]{caption}
\usepackage{subcaption}

\pgfplotsset{compat=newest}
\pgfplotsset{plot coordinates/math parser=false}

\usetikzlibrary{decorations.pathreplacing}
\usetikzlibrary{patterns}

\usetikzlibrary{external}
\newlength{\figwidth}
\newlength{\figheight}
\definecolor{gray1}{gray}{0.2}
\definecolor{gray2}{gray}{0.4}
\definecolor{gray3}{gray}{0.5}
\definecolor{gray4}{gray}{0.8}
\definecolor{gray5}{gray}{0.9}

\newlength{\prel}

\newcommand{\nodeset}{\mathcal{X}}
\newcommand{\rkhs}{\mathcal{H}}
\newcommand{\data}{\mathcal{D}}
\newcommand{\kmean}{k_\mu}
\newcommand{\sobolev}{W}
\newcommand{\difop}{\mathrm{D}}
\newcommand{\textsm}[1]{{\scriptscriptstyle\text{#1}}}
\newcommand{\her}{\mathrm{H}}

\newcommand{\TheTitle}{Fully symmetric kernel quadrature} 
\newcommand{\TheAuthors}{Toni Karvonen and Simo Särkkä}

\headers{\TheTitle}{\TheAuthors}

\title{{\TheTitle}\thanks{Submitted to arXiv on March 18, 2017. Revised on October 11, 2017 and on January 6, 2018. Accepted for publication in \emph{SIAM Journal on Scientific Computing} on January 3, 2018.
\funding{This work was supported by Aalto ELEC Doctoral School as well as Academy of Finland projects 266940 and 273475.}}}

\author{
  Toni Karvonen\thanks{Department of Electrical Engineering and Automation, Aalto University, Espoo, Finland (\email{toni.karvonen@aalto.fi}, \email{simo.sarkka@aalto.fi}).}
  \and
  Simo Särkkä\footnotemark[2]
}

\allowdisplaybreaks

\begin{document}

\maketitle

\begin{abstract} Kernel quadratures and other kernel-based approximation methods typically suffer from prohibitive cubic time and quadratic space complexity in the number of function evaluations. The problem arises because a system of linear equations needs to be solved. In this article we show that the weights of a kernel quadrature rule can be computed efficiently and exactly for up to tens of millions of nodes if the kernel, integration domain, and measure are fully symmetric and the node set is a union of fully symmetric sets. This is based on the observations that in such a setting there are only as many distinct weights as there are fully symmetric sets and that these weights can be solved from a linear system of equations constructed out of row sums of certain submatrices of the full kernel matrix. We present several numerical examples that show feasibility, both for a large number of nodes and in high dimensions, of the developed fully symmetric kernel quadrature rules. Most prominent of the fully symmetric kernel quadrature rules we propose are those that use sparse grids.
\end{abstract}



\section{Introduction}

Let $\Omega$ be a subset of $\R^d$, $\mu$ a measure on $\Omega$, and \sloppy{${f\colon\Omega \to \R}$} a function that is integrable with respect to $\mu$. Computation of the integral \sloppy{${\mu(f) \coloneqq \int_\Omega f \dif \mu}$} is a recurring problem in applied mathematics and statistics. In most cases, this integral has no readily available analytical form and one must resort to a quadrature rule (or, occasionally, a cubature rule if $d > 1$) for its approximation. A quadrature rule $Q$ is a linear functional of the form
\begin{equation*}
Q(f) \coloneqq \sum_{i=1}^n w_i f(\mx_i) \approx \int_\Omega f \dif \mu,
\end{equation*}
where $\mx_i \in \R^d$ are the nodes and $w_i \in \R$ are the weights. The nodes and weights are often chosen so that the quadrature approximation is exact whenever the integrand is a low-degree polynomial~\cite{DavisRabinowitz1984,Cools1997}---such methods are called classical or polynomial quadrature rules in this article (we reserve the term Gaussian for rules that use $n$ nodes to integrate polynomials up to degree $2n-1$ exactly). Another possibility is to use Monte Carlo or quasi Monte Carlo methods~\cite{Caflisch1998}.

Here we study \emph{kernel quadrature rules} that are, for arbitrary fixed nodes, optimal in the reproducing kernel Hilbert space (RKHS) induced by a user-specified positive-definite kernel. In this setting, optimality is measured in terms of the worst-case error (or, equivalently, the average-case error~\cite{Ritter2000,NovakWozniakowski2010vol2}). Kernel quadrature rules go back at least to the work of Larkin~\cite{Larkin1970,Larkin1972} in the 1970s. Lately, these rules have been a subject of renewed interest because they can be used for numerical integration on scattered data sets~\cite{Bezhaev1991,SommarivaVianello2006a} and they carry a probabilistic interpretation as posterior means for Gaussian processes assigned to the integrand~\cite{OHagan1991}. The probabilistic interpretation, equivalent to the RKHS formulation we use, is interesting because it allows for statistical modelling of error in numerical integration and is one of main motivators behind this article. The above topics, including the probabilistic interpretation, are reviewed in \Cref{sec:kernelquad}.

An advantage of kernel quadrature rules is that the nodes are \emph{not} prescribed as opposed to polynomial quadrature rules (polynomial rules with arbitrary nodes could probably be developed along the lines of de Boor and Ron interpolation~\cite{deBoorRon1990,NarayanXiu2012}). The flexibility comes with the price of having to solve the weights from a linear system of $n$ equations, a task of cubic time and quadratic space complexity. It would not be practical to tabulate the weights beforehand as they depend on the kernel, integration domain, and measure. Partially due to the computational cost, only low numbers of nodes have been used in kernel quadrature and much of the literature is concerned with optimal selection of the nodes. See~\cite{Larkin1970,OHagan1991,OHagan1992,Minka2000,SarkkaHartikainenSvenssonSandblom2016} for some optimal node configurations,~\cite{Oettershagen2017} for an algorithm to generate such nodes in one dimension, and~\cite{BriolOatesGirolamiOsborne2015,BriolOatesCockayneChenGirolami2017} for other non-optimal alternatives. Efficient computation of the optimal nodes in higher dimensions is an open problem and not the topic of this article. Instead, we want to find nodes for which the weights can be computed easily and fast.

There is not much work on extending applicability of kernel quadrature to integration problems where it is necessary to use a large number of nodes due to high dimensionality or high level of accuracy that is desired. O'Hagan~\cite{OHagan1991,OHagan1992} has proposed some computationally beneficial product grid (number of nodes grows quickly in dimension and when the grid is refined) and simplex (only $d+1$ nodes) designs that are too inflexible to be of much use in many situations. Most exciting work is by Oettershagen~\cite{Oettershagen2017} who has recently shown that the standard approach to sparse grid quadrature can be used to achieve quadratic time complexity. Furthermore, several fast and approximate kernel-based methods have been developed in scattered data approximation, statistics, and machine learning literature (a compendium can be found in, e.g.,~\cite[Supplement C]{BriolOatesGirolamiOsborneSejdinovic2016}). However, accuracy of quadrature rules is often strongly dependent on the weights having been computed exactly and approximate weights also give rise to some philosophical objections if they are to be used for statistical modelling of error of an integral approximation.

In this article we show that if certain structure is imposed on the node set, then the kernel quadrature weights can be computed \emph{exactly} and in a very simple manner. Our approach is based on \emph{fully symmetric sets}~\cite{Genz1986,GenzKeister1996} which are point sets that can be obtained from a given vector through permutations and sign changes of its coordinates. In \Cref{sec:fss} we show that some symmetricity assumptions on $\Omega$ and $\mu$ (see \Cref{ass:main}) lead, for a large class of kernels that includes all isotropic kernels, to tractable computation of the weights if the node set is a union of fully symmetric sets. Depending on the dimension, the weights can be computed for sets of this type that contain up to tens of millions of nodes. The crucial observation under our assumptions is that there are only as many distinct weights as there are fully symmetric sets making up the node set. The \emph{fully symmetric kernel quadratures} we construct exhibit the following advantageous properties:
\begin{itemize}
\item The algorithm (see \Cref{sec:fskq}) for exact computation of the weights is exceedingly simple and easy to implement.
\item If there are $J$ fully symmetric sets containing $n$ nodes in total, only $Jn$ kernel evaluations are needed. In all situations we can envision, $J$ is at most a few hundred while $n$ can, as mentioned, go up to millions (\Cref{sec:exapriori} contains an example where $J = 832$ and $n=15{,}005{,}761$). The weights are solved from a system of $J$ linear equations and only a $J \times J$ matrix needs to be stored.
\item The node selection scheme remains quite flexible and the number of nodes does not grow too fast with the dimensions (see \Cref{eq:FSScardinality} and \Cref{table:RGsize}) as happens when, for example, full Cartesian grids are used. The smallest non-trivial fully symmetric sets contain $2d$ points. 
\end{itemize}

In \Cref{sec:fsssel} we discuss a number of possible ways of selecting the fully symmetric sets. Sparse grids~\cite{BungartzGriebel2004}, popular in polynomial-based high-dimensional quadrature, are maybe the most obvious and promising choice. For kernel quadratures that use Clenshaw--Curtis sparse grids~\cite{NovakRitter1996} we also provide some theoretical convergence guarantees in \Cref{thm:SGconv}. Even though kernel quadrature rules on sparse grids can be efficiently constructed without the use of fully symmetric sets~\cite{Oettershagen2017}, our approach appears to be computationally competitive. In any case, sparse grids serve as a straightforward node selection scheme for showcasing that our algorithm indeed works.

The fast weight algorithm for computing the weights is not easily extended to fitting of the kernel parameters that often have considerable effect on accuracy of the integral approximation. Our experiments show that fully symmetric kernel quadratures are feasible but we have to resort to ad hoc solutions for fitting the kernel parameters or know them beforehand. Development of efficient fitting procedures is left for future research. This is discussed in \Cref{sec:lengthscale}.

Finally, it is worth remarking that this article is not the first instance of fully symmetric sets being used in conjunction with kernel quadrature. Arguably the simplest non-trivial fully symmetric kernel quadrature rule (this rule appears briefly in \Cref{sec:wcemin}) has seen use in approximate filtering of non-linear systems~\cite{SarkkaHartikainenSvenssonSandblom2016,PrueherStraka2017,PruherEtal2017}, but without an efficient weight computation algorithm.

\section{Kernel quadrature}\label{sec:kernelquad}

This section reviews the basics of quadrature rules in reproducing kernel Hilbert spaces. We also briefly discuss connections to probabilistic modelling of numerical algorithms. See~\cite{Larkin1970,BriolOatesGirolamiOsborneSejdinovic2016,Oettershagen2017} for proofs and additional references. Standard references on reproducing kernel Hilbert spaces are~\cite{Aronszajn1950,BerlinetThomasAgnan2004}.

\subsection{Quadrature in reproducing kernel Hilbert spaces}

A kernel \sloppy{${k \colon \Omega \times \Omega \to \R}$} is said to be positive-definite if the $n \times n$ kernel Gram matrix $[\mK]_{ij} \coloneqq k(\mx_i, \mx_j)$ is positive-definite for every $n \geq 0$ and any distinct $\mx_1,\ldots,\mx_n \in \Omega$. Every continuous positive-definite kernel defines a unique reproducing kernel Hilbert space $\rkhs$ of functions $f\colon\Omega\to\R$ through the properties (i) $k(\cdot,\mx)\in\rkhs$ for every $\mx\in\Omega$ and (ii) pointwise evaluations of any $f\in\rkhs$ can be represented in terms of inner product with the kernel: $\inprod{f}{k(\cdot,\mx)}_\rkhs = f(\mx)$. The latter of these is called the reproducing property. The integral operator $\mu$ and the quadrature rule $Q$ are bounded linear functionals on $\rkhs$ under the non-restrictive assumption $\int_\Omega \sqrt{k(\mx,\mx)} \dif\mu(\mx) < \infty$. The worst-case error (WCE) $e(Q)$ of a quadrature rule $Q$ is defined in terms of the dual norm
\begin{equation}\label{eq:RKHSerror}
e(Q) \coloneqq \norm[0]{\mu - Q}_{\rkhs^*} = \sup_{\norm[0]{f}_\rkhs \leq 1} \abs[0]{\mu(f) - Q(f)}
\end{equation}
that can be also written as $e(Q) = \norm[0]{\mu[k(\cdot,\mx)] - Q[k(\cdot,\mx)]}_\rkhs$. Why this is a reasonable measure of error of the quadrature rule is apparent after an application of the reproducing property and the Cauchy--Schwarz inequality:
\begin{equation*}
\abs[0]{\mu(f) - Q(f)} = \abs[1]{\inprod{f}{\mu[k(\cdot,\mx)] - Q[k(\cdot,\mx)]}_\rkhs} \leq e(Q)\norm[0]{f}_\rkhs
\end{equation*}
for $f \in \rkhs$. That is, if the integrand belongs to the RKHS, convergence in the usual sense of diminishing integration error is implied by convergence to zero of the WCE. Relationship between the kernel and its induced RKHS is further discussed in \Cref{sec:convergence} where we also provide two convergence theorems for the worst-case error.

The quadrature rule that, for arbitrary fixed distinct nodes $\nodeset = \{\mx_1,\ldots,\mx_n\}$, minimises the worst-case error~\eqref{eq:RKHSerror} is called the kernel quadrature rule and denoted by $Q_k$. This rule is unique and the optimal weights $\mw = (w_1,\ldots,w_n)$ can be solved from
\begin{equation}\label{eq:KQweights}
\underbrace{\begin{pmatrix} k(\mx_1,\mx_1) & \cdots & k(\mx_1,\mx_n) \\ \vdots & \ddots & \vdots \\ k(\mx_n,\mx_1) & \cdots & k(\mx_n,\mx_n) \end{pmatrix}}_{=\mK} \begin{pmatrix} w_1 \\ \vdots \\ w_n\end{pmatrix} = \underbrace{\begin{pmatrix} \kmean(\mx_1) \\ \vdots \\ \kmean(\mx_n) \end{pmatrix}}_{=\kmean(\nodeset)},
\end{equation}
where the kernel Gram matrix $\mK$ is positive-definite---and hence non-singular---and $\kmean(\mx) \coloneqq \int_\Omega k(\mx,\mx')\dif\mu(\mx')$ is the kernel mean, an object of much independent interest~\cite{MuandetFukumizuSriperumbuduScholkopf2017}. The kernel quadrature rule and its worst-case error are
\begin{align}
Q_k(f) &= \sum_{i=1}^n [\mK^{-1}\kmean(\nodeset)]_i f(\mx_i) = \my^\transpose\mK^{-1}\kmean(\nodeset), \nonumber \\
e(Q_k)^2 &= \int_\Omega\int_\Omega k(\mx,\mx') \dif\mu(\mx)\dif\mu(\mx') - \kmean(\nodeset)^\transpose \mK^{-1}\kmean(\nodeset) = \mu(\kmean) - Q_k(\kmean), \label{eq:wce}
\end{align}
where $\my = (f(\mx_1),\ldots,f(\mx_n))$. 

An \emph{optimal kernel quadrature rule} minimises the worst-case error also over the node set (of fixed cardinality). Such rules cannot be constructed efficiently at the moment in dimensions larger than one. We discuss structurally constrained versions in \Cref{sec:wcemin}.

\subsection{Probabilistic interpretation}\label{sec:probnum}

The probabilistic interpretation of kernel quadrature as \emph{Bayesian quadrature} is a part of the emergent field \emph{probabilistic numerical computing}~\cite{Diaconis1988,OHagan1992,HennigOsborneGirolami2015,CockayneOatesSullivanGirolami2017}, origins of which can be traced at least back to the work of Larkin~\cite{Larkin1972}. This interpretation is a major motivator behind the present article.

In Bayesian quadrature, the integrand $f$ is typically modelled as a Gaussian process~\cite{OHagan1978,RasmussenWilliams2006} (prompting the alternative term \emph{Gaussian process quadrature}) with the covariance kernel $k$. With the node locations $\mx_1,\ldots,\mx_n$ and function evaluations $f(\mx_1),\ldots,f(\mx_n)$ considered the ``data'' $\data$, the posterior $f \mid \data$ is a Gaussian process with the mean and covariance
\begin{align*}
\expec[f(\mx) \mid \data] &= \my^\transpose \mK^{-1}k(\nodeset, \mx), \\ 
\mathbb{C}[f(\mx),f(\mx') \mid \data] &= k(\mx, \mx') - k(\mx, \nodeset)^\transpose \mK^{-1} k(\mx, \nodeset),
\end{align*}
where $[k(\mx,\nodeset)]_i = k(\mx,\mx_i)$. Because $\mu$ is a linear operator, this induces the Gaussian posterior distribution $\mu(f) \mid \data$ on the integral with the mean $\expec[\mu(f) \mid \data]$ and variance $\mathbb{V}\big[\mu(f) \mid \data\big]$ that turn out to be precisely $Q_k(f)$ and $e(Q_k)^2$ from the preceding section. The worst-case error can be therefore interpreted as a measure of numerical uncertainty over the integral approximation and then exploited in for instance uncertainty quantification and allocation of limited computational resources in computational pipelines~\cite{CockayneOatesSullivanGirolami2017}. Clear expositions of this probabilistic viewpoint to numerical integration are~\cite{OHagan1991,Minka2000,BriolOatesGirolamiOsborneSejdinovic2016} and the methodology is quite popular in machine learning (see, e.g.,~\cite{RasmussenGhahramani2002,Gunter2014}).

\section{Fully symmetric kernel quadrature}\label{sec:fss}

This is the main section of the article. We introduce fully symmetric sets, their connection to multivariate quadrature rules and prove our main result, \Cref{thm:main}, on computational benefits of doing kernel quadrature with node sets that are unions of fully symmetric sets.

\subsection{Fully symmetric sets}

\begin{figure}[t]
\centering
  \begin{subfigure}[b]{0.44\textwidth}
  \centering
  \includegraphics{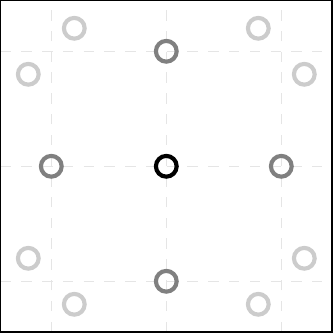}
  \end{subfigure}
  \begin{subfigure}[b]{0.44\textwidth}
  \centering
  \includegraphics{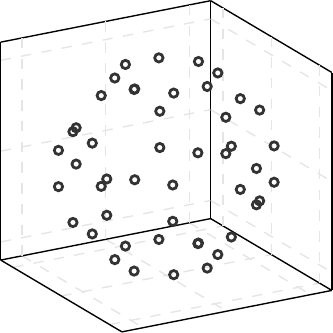}
  \end{subfigure}
  \caption{Examples of fully symmetric sets in two and three dimensions. Left: the fully symmetric sets $[0,0]$, $[1,0]$, and $[1.2,0.8]$ in $\R^2$. Right: the fully symmetric set $[1, 0.5, 0.2]$ that consists of 48 elements in $\R^3$.}\label{fig:FSS}
\end{figure}

\begin{table}[t]
\begin{center}
\captionof{table}{Cardinalities, as computed from \Cref{eq:FSScardinality}, of fully symmetric sets generated by $m = 1,\ldots,9$ distinct non-zero generators for dimensions $d=2,\ldots,9$.}\label{table:RGsize}
\small
\caption*{\small{\textbf{Dimension}}}\vspace{-0.3cm}
\begin{tabular}[t!]{c|c c c c c c c c c c}
\arrayrulecolor{gray4}
$m$ & 2 & 3 & 4 & 5 & 6 & 7 & 8 & 9 \\
\Xhline{1pt}
1 & 4 & 6 & 8 & 10 & 12 & 14 & 16 & 18 \\ \hline
2 & 8 & 24 & 48 & 80 & 120 & 168 & 224 & 288\\ \hline
3 & & 48 & 192 & 480 & 960 & 1,680 & 2,688 & 4,032\\ \hline
4 & & & 384 & 1,920 & 5,760 & 13,440 & 26,880 & 48,384\\ \hline
5 & & & & 3,840 & 23,040 & 80,640 & 215,040 & 483,840\\ \hline
6 & & & & & 46,080 & 322,560 & 1,290,240 & 3,870,720 \\ \hline
7 & & & & & & 645,120 & 5,160,960 & 23,224,320\\ \hline
8 & & & & & & &  10,321,920 & 92,897,280\\ \hline
9 & & & & & & & & 185,794,560 \\
\end{tabular}
\end{center}
\end{table}

A fully symmetric set is a point set in $\R^d$ that is obtained from a given vector through permutations and sign changes of its coordinates. Let $\Pi_d$ be the set of all permutations $\mq = (q_1,\ldots,q_d)$ of the integers $1,\ldots,d$ and $S_d$ the set of all vectors of the form $\ms = (s_1,\ldots,s_d)$ with each $s_i$ either $1$ or $-1$. Then, given $d$ non-negative scalars $\mlambda = (\lambda_1,\ldots,\lambda_d)$ called \emph{generators}, the point set
\begin{equation}\label{eq:FSS}
[\mlambda] = [\lambda_1,\ldots,\lambda_d] \coloneqq \bigcup_{\mq \in \Pi_{d}}\bigcup_{\ms \in S_d} \big\{(s_1\lambda_{q_1},\ldots,s_d\lambda_{q_d}) \big\} \subset \R^d
\end{equation}
is the fully symmetric set generated by the \emph{generator vector} $\mlambda$. With $m$ the number of non-zero generators, $m_0$ the number of zero generators (i.e.\ $m = d - m_0$), and $m_1,\ldots,m_l$ multiplicities of distinct non-zero generators so that $\sum_{i=1}^l m_i = m$, cardinality of the fully symmetric set~\eqref{eq:FSS} is
\begin{equation}\label{eq:FSScardinality}
\#[\lambda_1,\ldots,\lambda_d] = \frac{2^m d!}{m_0! \cdots m_l!}.
\end{equation}

An alternative way of writing \Cref{eq:FSS} is via \emph{permutation matrices} as $[\mlambda] = \bigcup_{\mP} \mP\mlambda$, where the the union is over all $d\times d$ permutation and sign change matrices $\mP$. These are matrices that have on each row and column exactly one element that is either $1$ or $-1$ and the rest are zero. Any element of a fully symmetric set can be obtained from any other via linear transformation by an appropriate permutation matrix. How \Cref{eq:FSS} works and what the resulting point sets look like is illustrated in two and three dimensions in \Cref{ex:FSS} and \Cref{fig:FSS}. Note that all elements of a fully symmetric set are equidistant from the origin which is to say that if $\mx \in [\lambda_1,\ldots,\lambda_d]$, then $\norm[0]{\mx}^2 = \norm[0]{\mlambda}^2 = \lambda_1^2 + \cdots + \lambda_d^2$. We also need the concept of a fully symmetric function.
\begin{definition} A function $f\colon\Omega \to \R$ is \emph{fully symmetric} if it is constant in every fully symmetric set. That is, with $\mlambda$ any generator vector, it holds that $f(\mx) = f(\mx')$ for any $\mx,\mx'\in\Omega \cap [\mlambda]$. Alternatively, $f(\mP\mx) = f(\mx)$ for any $\mx\in\Omega$ and any permutation and sign change matrix $\mP$ such that $\mP\mx \in \Omega$.
\end{definition}

\begin{example}\label{ex:FSS} In $\R^3$, the non-zero and distinct generators $\lambda_1$ and $\lambda_2$ generate the fully symmetric set
\begin{equation*}
[\lambda_1,\lambda_2,0] = 
\begin{subequations}
\begin{aligned}
\big\{&(\lambda_1,\lambda_2,0), && (-\lambda_1, \lambda_2, 0), && (\lambda_1,-\lambda_2,0), && (-\lambda_1,-\lambda_2,0),\big.\\
&(\lambda_2,\lambda_1,0), && (-\lambda_2,\lambda_1,0), && (\lambda_2,-\lambda_1,0), && (-\lambda_2,-\lambda_1,0), \\
&(0,\lambda_1,\lambda_2), && (0,-\lambda_1,\lambda_2), && (0,\lambda_1,-\lambda_2), && (0,-\lambda_1,-\lambda_2), \\
&(0,\lambda_2,\lambda_1), && (0,-\lambda_2,\lambda_1), && (0,\lambda_2,-\lambda_1), && (0,-\lambda_2,-\lambda_1), \\
&(\lambda_1,0,\lambda_2), && (-\lambda_1,0,\lambda_2), && (\lambda_1,0,-\lambda_2), && (-\lambda_1,0,-\lambda_2), \\
&\big. (\lambda_2,0,\lambda_1), && (-\lambda_2,0,\lambda_1), && (\lambda_2,0,-\lambda_1), && (-\lambda_2,0,-\lambda_1) \big\}
\end{aligned}
\end{subequations}
\end{equation*}
that has $2^2\times 3!/(1!\times1!\times1!) = 24$ elements. In terms of permutation matrices, the element $(-\lambda_1,0,\lambda_2)$ is
\begin{equation*}
\begin{pmatrix} -\lambda_1 \\ 0 \\ \lambda_2 \end{pmatrix} = \begin{pmatrix} -1 & 0 & 0 \\ 0 & 0 & 1 \\ 0 & 1 & 0\end{pmatrix} \begin{pmatrix} \lambda_1 \\ \lambda_2 \\ 0 \end{pmatrix}.
\end{equation*}
\end{example}

The method we have used to generate fully symmetric sets out of user-specified generator vectors is detailed in Algorithm 1 in \Cref{sec:fskq}. There are many other possibilities; we do not claim that the one presented is the optimal implementation.

\subsection{Fully symmetric quadrature rules}\label{sec:fsqr}

The notation $f[\mlambda] = f[\lambda_1,\ldots,\lambda_d]$ stands for the sum of evaluations of $f$ at the points of the fully symmetric set:
\begin{equation*}
f[\mlambda] \coloneqq \sum_{\mx \in [\mlambda]} f(\mx).
\end{equation*}
A \emph{fully symmetric quadrature rule} is a quadrature rule of the form
\begin{equation*}
Q(f) = \sum_{\mlambda\in\Lambda} w_{\mlambda} f[\mlambda] = \sum_{\mlambda \in \Lambda} w_{\mlambda} \sum_{\mx \in [\mlambda]} f(\mx),
\end{equation*}
where $\Lambda$ is a given finite collection of distinct generator vectors $\mlambda$. Such a rule uses only $\#\Lambda$ distinct weights, each corresponding to often a very large number of nodes. In \Cref{thm:main} we establish conditions under which a kernel quadrature rule is fully symmetric. This will yield significant computational savings because only $\#\Lambda$ (instead of $n = \sum_{\mlambda \in \Lambda} \#[\mlambda]$) distinct weights need to be computed.

Fully symmetric quadrature rules are prominent among classical polynomial quadrature rules, work on them going back to~\cite{Lyness1965a,McNameeStenger1967}. To the best of our knowledge, the most general and efficient constructions have been given by Genz~\cite{Genz1986} for the uniform distribution on a square and by Genz and Keister~\cite{GenzKeister1996} for Gaussians on the whole real space (a case studied also in~\cite{LuDarmofal2004}). See, for example, the review~\cite{Cools1997} for more examples and discussion on the highly related invariant theory. To achieve high algebraic order of precision, the classical fully symmetric quadrature rules rely on symmetry of the underlying measure and advantageous properties of polynomials when integrated with respect to such measures. In contrast to the kernel quadrature rules we are about to construct, the aforementioned rules do not permit free selection of the fully symmetric sets that are to be used.

Many of the popular sparse grid rules are also fully symmetric~\cite{NovakRitterSchmittSteinbauer1999,NovakRitter1999}. We make use of this useful fact in \Cref{sec:smolyak} for construction of sparse grid kernel quadrature rules whose weights can be computed efficiently.

\subsection{Fully symmetric kernels}\label{sec:fskernels}

We can now introduce the class of kernels that this article is concerned with as well as the necessary assumptions on the integration domain and measure.

\begin{definition} Suppose that $\mP\mx \in \Omega$ for any $\mx \in \Omega$ and any permutation and sign change matrix $\mP$. A kernel $k$ is \emph{fully symmetric} if $k(\mP\mx,\mP\mx') = k(\mx,\mx')$ for any such matrix $\mP$ and any $\mx,\mx' \in \Omega$.
\end{definition}

This class of kernels includes (i) isotropic kernels, (ii) products of an isotropic kernel $k_1$ of the form $k(\mx,\mx') = \prod_{i=1}^d k_1(\abs[0]{x_i-x_i'})$, and (iii) sums of an isotropic kernel $k_1$ of the form $k(\mx,\mx') = \sum_{i=1}^d k_1(\abs[0]{x_i-x_i'})$. Some polynomial kernels\footnote{Strictly speaking, these kernels are not positive-definite as the kernel matrix does not remain non-singular for any number of distinct points. See also \Cref{remark:posdef}.} of the form $k(\mx,\mx') = \sum_{i=1}^p P_i(\mx)P_i(\mx')$ for suitable multivariate polynomials $P_i$ are also fully symmetric. For example, the selection $P_1 \equiv 1$ and $P_i = x_{i-1}^2$ for $i=2,\ldots,p=d+1$ results in a fully symmetric kernel. See~\cite{SarkkaHartikainenSvenssonSandblom2016,KarvonenSarkka2017} for some results on how quadrature rules for such kernels are related to classical quadrature rules.

\begin{assumption}\label{ass:main} We assume that
\begin{enumerate}
\item[(i)] The integration domain $\Omega \subset \R^d$ is invariant under permutations and sign changes of coordinates of its elements. That is, $\Omega = \mP \Omega = \Set{ \mP\omega}{\omega \in \Omega}$ for any permutation and sign change matrix $\mP$.
\item[(ii)] The measure $\mu$ is fully symmetric in the sense that its density $f_\mu$ (w.r.t.\ the Lebesgue measure) is a fully symmetric function.
\item[(iii)] The kernel $k$ is positive-definite and fully symmetric.
\end{enumerate}
\end{assumption}
This assumption holds, for example, for $\Omega = [-1,1]^d$ equipped with the uniform measure and $\Omega = \R^d$ equipped with the Gaussian measure as well as for many other cases of interest. The numerical examples in \Cref{sec:num} are for these two cases.

\begin{lemma}\label{lemma:kmean} The kernel mean $k_\mu$ is fully symmetric under \Cref{ass:main}.
\end{lemma}
\begin{proof} Let $\mx$ and $\mx'$ be elements of the same fully symmetric set. That is, $\mx' = \mP\mx$ for some permutation and sign change matrix $\mP$. A change of variables yields
\begin{equation*}
\int_\Omega k(\mx',\mz) f_\mu(\mz) \dif \mz = \int_\Omega \abs[0]{\det \mP}k\big(\mP\mx,\mP\mz) q\big(\mP\mz\big)\dif \mz = \int_\Omega k(\mx,\mz) f_\mu(\mz) \dif \mz,
\end{equation*}
where we have used the fact that $\abs[0]{\det \mP} = 1$. That is, $\kmean(\mx') = \kmean(\mx)$.
\end{proof}

\subsection{Fully symmetric kernel quadrature}\label{sec:fskq}

Let $[\mlambda^1], \ldots, [\mlambda^J]$ be distinct fully symmetric sets generated by $\mlambda^1,\ldots,\mlambda^J \in \R^d$. If the node set $\nodeset$ is the union \sloppy{${\nodeset = \cup_{j=1}^J [\mlambda^j]}$} of these fully symmetric sets, then a kernel quadrature rule using this node set is a fully symmetric quadrature rule in the sense of \Cref{sec:fsqr}. Furthermore, its $J$ distinct weights can be computed extremely efficiently when compared to naively solving the linear system~\eqref{eq:KQweights} of $\#\nodeset = n$ equations. This is formalised in the following theorem. \Cref{fig:reduction} illustrates the simplified weight computation process in the case of a node set that is a union of three fully symmetric sets.

\begin{figure}[t]
\centering
  \scalebox{0.5}{\includegraphics{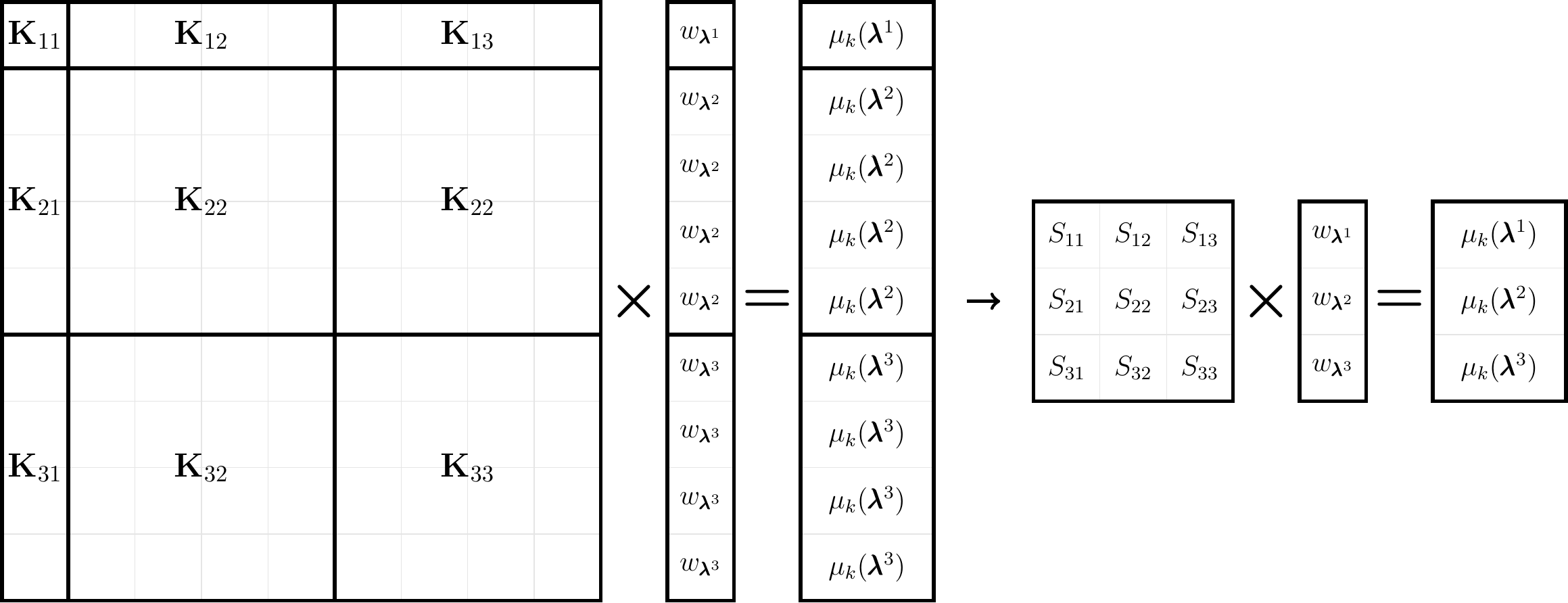}}
\caption{Illustration of \Cref{thm:main} for a node set that is a union of three fully symmetric sets: one containing one element and two containing four elements. All row sums of the matrices $\mK_{ij}$, defined in \Cref{eq:blockmatrix}, are equal to $S_{ij}$.}\label{fig:reduction}
\end{figure}

\begin{theorem}\label{thm:main} Suppose that $\Omega$, $\mu$, and $k$ satisfy \Cref{ass:main}. If the node set $\nodeset$ is a union of $J$ distinct fully symmetric sets $[\mlambda^1],\ldots,[\mlambda^J]$, then the kernel quadrature rule $Q_k$ is fully symmetric:
\begin{equation*}
Q_k(f) = \sum_{j=1}^J w_{\mlambda^j} f[\mlambda^j].
\end{equation*}
Furthermore, the $J$ weights $w_{\mlambda^1},\ldots,w_{\mlambda^J}$ corresponding to the fully symmetric sets can be solved from the non-singular linear system of $J$ equations
\begin{equation}\label{eq:reducedSystem}
\begin{pmatrix} S_{11} & \cdots & S_{1J} \\ \vdots & \ddots & \vdots \\ S_{J1} & \cdots & S_{JJ} \end{pmatrix} \begin{pmatrix} w_{\mlambda^1} \\ \vdots \\ w_{\mlambda^J}\end{pmatrix} = \begin{pmatrix} \kmean(\mlambda^1) \\ \vdots \\ \kmean(\mlambda^J)\end{pmatrix},
\end{equation}
where
\begin{equation*}
S_{ij} = \sum_{\mx\in[\mlambda^j]} k(\mx^i,\mx) \quad \text{for any} \quad \mx^i \in [\mlambda^i].
\end{equation*}
\end{theorem}

\begin{proof}
Let $\nodeset = \cup_{j=1}^J [\mlambda^j]$ be ordered such that all elements of a single fully symmetric set appear consecutively and the fully symmetric sets themselves are in ascending order in terms of their index $j$.

We denote $n^i = \#[\mlambda^i]$ and enumerate each fully symmetric set as \sloppy{${[\mlambda^i] = \{\mx_1^i, \ldots, \mx_{n^i}^i\}}$}. By \Cref{lemma:kmean}, the kernel mean is fully symmetric and, consequently, the kernel mean vector $\kmean(\nodeset) \in \R^n$, $n = n^1 + \cdots + n^J$, is
\begin{equation*}
\kmean(\nodeset) = \big( \kmean([\mlambda^1]), \ldots, \kmean([\mlambda^J]) \big),
\end{equation*}
where $\kmean([\mlambda^j]) = (\kmean(\mlambda^j), \ldots, \kmean(\mlambda^j)) \in \R^{n^j}$. That is, $\kmean(\nodeset)$ contains only $J$ distinct elements that occur in blocks of $n^j$. Consider then the kernel matrix $\mK$ that can be partitioned into $J^2$ submatrices $\mK_{ij}$ of dimensions $n^i \times n^j$, each containing all the kernel evaluations $k(\mx^i,\mx^j)$ for $\mx^i\in[\mlambda^i]$ and $\mx^j \in [\mlambda^j]$:
\begin{equation}\label{eq:blockmatrix}
\mK = \begin{pmatrix} \mK_{11} & \cdots & \mK_{1J} \\ \vdots & \ddots & \vdots \\ \mK_{J1} & \cdots & \mK_{JJ} \end{pmatrix}, \enspace \text{where} \enspace \mK_{ij} = \begin{pmatrix} k(\mx_1^i,\mx_1^j) & \cdots & k(\mx_1^i,\mx_{n^j}^j) \\ \vdots & \ddots & \vdots \\ k(\mx_{n^i}^i,\mx_1^j) & \cdots & k(\mx_{n^i}^i,\mx_{n^j}^j) \end{pmatrix}.
\end{equation}
Any row of any submatrix $\mK_{ij}$ can be obtained from any other of its rows by a permutation of elements of the row. To confirm this, consider any distinct rows $p, p' \leq n^i$ of $\mK_{ij}$. There exists a permutation and sign change matrix $\mP$ such that $\mx_{p'}^i = \mP\mx_p^i$ because fully symmetric sets are closed under such transformations. Note that $\mP$ is non-singular and its inverse $\mP^{-1}$ is also a permutation and sign change matrix. Since the kernel is fully symmetric, for any $l \leq n^j$ we have
\begin{equation*}
k\big(\mx_p^i, \mx_l^j\big) = k\big(\mP^{-1} \mx_{p}^i, \mP^{-1} \mx_l^j\big) = k\big(\mx_{p'}^i, \mP^{-1} \mx_l^j\big),
\end{equation*}
where $\mP^{-1} \mx_l^j \in [\mlambda^j]$. This means that for every $l$ there is an element on the row $p'$ that equals the $l$th element of the $p$th row. That is, the rows are permutations of each other. Consequently, the row sums $S_{ij} \coloneqq \sum_{\mx \in [\mlambda^j]} k(\mx^i,\mx)$ of $\mK_{ij}$ do not depend on $\mx^i \in [\mlambda^i]$.

Consider the $J \times J$ matrix $[\mS]_{ij} = S_{ij}$ composed of the row sums defined above. Then the equation $\mS\ma = \mb$ for some vectors $\ma,\mb \in \R^J$ implies that
\begin{equation*}
\begin{pmatrix} \mK_{11} & \cdots & \mK_{1J} \\ \vdots & \ddots & \vdots \\ \mK_{J1} & \cdots & \mK_{JJ} \end{pmatrix} \begin{pmatrix} \ma_1 \\ \vdots \\ \ma_J \end{pmatrix} = \begin{pmatrix} \mb_1 \\ \vdots \\ \mb_J \end{pmatrix},
\end{equation*}
where $\ma_i = (a_i, \ldots, a_i) \in \R^{n^i}$ and $\mb_i = (b_i,\ldots,b_i) \in \R^{n^i}$, because
\begin{equation*}
b_i = \sum_{j=1}^J a_j S_{ij} = \sum_{j=1}^J a_j \sum_{l=1}^{n^j} [\mK_{ij}]_{pl} = \sum_{j=1}^J a_j \sum_{l=1}^{n^j} k\big( \mx_p^i, \mx_l^j \big)
\end{equation*}
for every $p \leq n^i$. The matrix $\mS$ is non-singular, for if it were singular there would exist a non-zero vector $\ma \in \R^J$ such that $\mS\ma = \mZero$. But by the above argument this would imply that $\mK$ is singular which is not the case because the kernel $k$ is positive-definite. All this implies that if $(w_{\mlambda^1},\ldots,w_{\mlambda^J})$ is the unique solution to the linear system of equations~\eqref{eq:reducedSystem}, then 
\begin{equation*}
\mw = ( \mw_{\mlambda^1}, \ldots, \mw_{\mlambda^J}) \in \R^n, \hspace{0.5cm} \text{where} \hspace{0.5cm} \mw_{\mlambda^j} = (w_{\mlambda^j}, \ldots, w_{\mlambda^j} ) \in \R^{n^j},
\end{equation*}
must be the solution to $\mK\mw = \kmean(\nodeset)$. That is, weights for nodes in each fully symmetric set are equal and the kernel quadrature rule is fully symmetric. This concludes the proof.
\end{proof}

\begin{remark}\label{remark:posdef} \Cref{thm:main} also applies to kernels whose kernel matrix is positive-definite only for every collection of $m \leq p$ distinct points for some $p > 0$ if the total number $n$ of nodes does not exceed $p$. The polynomial kernels briefly mentioned in \Cref{sec:fskernels} are examples of such kernels.
\end{remark}

The full algorithm for fully symmetric kernel quadrature is presented in high-level pseudocode in Algorithm 1 below. We expect that $J$, the number of fully symmetric sets, is rarely more than a few hundred (the example in \Cref{sec:exapriori} has $J = 832$ but this results in \sloppy{${n \approx 15,000,000}$}) so solving the weights from the linear system~\eqref{eq:reducedSystem} is not a computational bottleneck. Instead, it is usually the $Jn$ kernel evaluations that take the most time.

\begin{algorithm}
\vspace{0.1cm}
\noindent \textbf{Algorithm 1} Fully symmetric kernel quadrature
\vspace{0.1cm}
\hrule
\vspace{0.1cm}
\noindent \emph{Construct the fully symmetric sets}
\begin{enumerate}
\item Select $J$ distinct generator vectors $\mlambda^j \in \R^d$ with non-negative elements.
\item[] \textbf{For each} $j=1,\ldots,J$ construct the fully symmetric set $[\mlambda^j]$:
\begin{enumerate}
\item[2.] Sort $\mlambda^j$ in descending order.
\item[3.] Identify the unique non-zero elements $\muu \in \R^{d_\muu}$, $d_\muu \leq d$, of $\mlambda^j$ and their multiplicities $\mm \in \N^{d_\muu}$. Denote $\Sigma_\mm = \sum_{l=1}^{d_\muu} m_l$. That is,
\begin{equation*}
\mlambda^j = \big(\widetilde{\muu}_1, \ldots, \widetilde{\muu}_{d_\muu},\mZero_{(d - \Sigma_\mm) \times 1}\big) \in \R^d,
\end{equation*}
where $\widetilde{\muu}_l = (u_l,\ldots,u_l) \in \R^{m_l}$ for $l=1,\ldots,d_\muu$.
\item[4.] Construct all $d_\ma$ possible vectors $\ma^i \in \N^{d_\muu}$ such that $a_l^i \leq m_l$ for each $l=1,\ldots,d_\muu$.
\item[5.] Set $[\mlambda^j] = \emptyset$.
\item[] \textbf{For each} $i = 1,\ldots,d_\ma$:
\begin{enumerate}
\item[6.] Construct the vector
\begin{equation*}
\mlambda^j_i = \big(\widetilde{\muu}_1^s, \ldots, \widetilde{\muu}_{d_\muu}^s,\mZero_{(d - \Sigma_\mm) \times 1}\big) \in \R^d,
\end{equation*}
where, for each $l=1,\ldots,d_\muu$, the first $a_l^i$ elements of $\widetilde{\muu}_l^s \in \R^{m_l}$ are $-u_l$ and the rest are $u_l$. The vector $\mlambda^j_i$ essentially corresponds to one possible sign combination $(s_1 \lambda_1^j,\ldots, s_d \lambda_d^j)$ in \Cref{eq:FSS}.
\item[7.] Compute the collection $U$ of all unique permutations of $\mlambda^j_i$ and append it to the fully symmetric set: $[\mlambda^j] = [\mlambda^j] \cup U$.
\end{enumerate}
\end{enumerate}
\end{enumerate}
\noindent\emph{Compute the kernel quadrature weights}
\begin{enumerate}
\item[8.] Construct an empty matrix $\mS \in \R^{J \times J}$.
\item[] \textbf{For each} $(i,j) \in \{1,\ldots,J\}^2$:
\begin{enumerate}
\item[9.] Select any $\mx^i \in [\mlambda^i]$ and set $[\mS]_{ij} = \sum_{\mx \in [\mlambda^j]} k(\mx^i, \mx)$.
\end{enumerate}
\item[10.] Solve the $J$ distinct weights $\mw_{\mlambda} = (w_{\mlambda^1},\ldots,w_{\mlambda^J})$ from the linear system of equations $\mS \mw_{\mlambda} = \mb$, where $b_l = \kmean(\mlambda^l)$.
\end{enumerate}
\noindent\emph{Compute the quadrature approximation}
\begin{enumerate}
\item[11.] Compute $Q_k(f) \approx \int_\Omega f \dif \mu$ as
\begin{equation*}
Q_k(f) = \sum_{j=1}^J w_{\mlambda^j} f[\mlambda^j] = \sum_{j=1}^J w_{\mlambda^j} \sum_{\mx \in [\mlambda^j]} f(\mx).
\end{equation*}
\end{enumerate}
\vspace{0.1cm}
\end{algorithm}

\section{Selection of the fully symmetric sets}\label{sec:fsssel}

This section presents three different approaches for constructing the node set as a union of fully symmetric sets. Of these the sparse grids of \Cref{sec:smolyak} are the most promising alternative. We also discuss convergence properties of some of the kernel quadrature rules we construct in \Cref{sec:convergence}. We expect there to exist many other competitive schemes as one of the main advantages of fully symmetric kernel quadrature is that there are no restrictions in selecting the generator vectors.

\subsection{Random generators}\label{sec:fskmcq}

Arguably, the simplest approach, both conceptually and algorithmically, is to draw a number of generator vectors randomly from the underlying distribution. However, unless additional constraints are enforced, all the generators will be distinct and non-zero, resulting in unrealistic numbers of integrand evaluations needed if $d \geq 6$, as seen from \Cref{table:RGsize}. One could heuristically set some generators to zero to reduce the number of nodes but it is not entirely clear how this should be done. In any case, for at least $d < 6$, the random generator approach seems realistic. We call this method the \emph{fully symmetric kernel Monte Carlo} (FSKMC).

\Cref{thm:RGconv} provides theoretical convergence guarantees and \Cref{sec:exrand} demonstrates that the FSKMC can also work in practice. Nevertheless, the method does not seem very promising as it comes across that a large number of random generator vectors, and thus an even larger number of nodes, is required to capture the underlying distribution.

This approach bears some similarity to stochastic radial and spherical integration rules developed in~\cite{GenzMonahan1998,GenzMonahan1999}. These rules are less flexible due to the usual constraints of integrating low-degree polynomials exactly and more involved in their implementation.

\subsection{Sparse grids}\label{sec:smolyak}

An iterated quadrature rule of degree $m$ based on a regular Cartesian product grid requires $m^d$ nodes---a number that quickly becomes impractically large. Sparse grids that originate in the work of Smolyak~\cite{Smolyak1963} are ``sparsified'' product sets widely used in numerical integration~\cite{NovakRitter1996,GerstnerGriebel1998,NovakRitter1999,NovakRitterSchmittSteinbauer1999}. See also the general survey by Bungartz and Griebel~\cite{BungartzGriebel2004} and~\cite{Holtz2011} for a wealth of financial applications. Recently, Oettershagen~\cite{Oettershagen2017} has shown that the standard approach to sparse grids is also applicable to fast computation of the weights of kernel quadrature rules. This approach is different from ours, that is based on identifying the fully symmetric sets a sparse grid is a union of, and specific to sparse grids. Other sparse grid based kernel methods appear in~\cite{Garcke2006,GeorgoulisLevesleySubhan2013,DongGeorgoulisLevesleyUsta2015,UstaLevesley2017}. The construction of sparse grids that we present in this section is not the most general possible as we work in the fully symmetric framework. More general constructions are contained in some of the aforementioned references. We assume that $\Omega = [-a,a]^d$ for a possibly infinite $a > 0$.

Let $X^1 = \{0\}$ and $X^i \subset X^{i+1} \subset [-a,a]$ for $i > 1$ be finite, nested and symmetric (i.e.\ if $x \in X^i$, then $-x \in X^i$) point sets. Then the \emph{sparse grid} of \emph{level} $q \geq 1$ is the set  
\begin{equation*}
H(q,d) \coloneqq \bigcup_{\abs[0]{\alpha} = d+q} \big( X^{\alpha_1} \times \cdots \times X^{\alpha_d} \big),
\end{equation*}
where $\alpha \in \N^d$ is a $d$-dimensional multi-index with the elements $\alpha_i = \alpha(i)$ and \sloppy{${\abs[0]{\alpha} = \alpha_1 + \cdots + \alpha_d}$}. Note that the largest $X^i$ that is needed for a sparse grid of level $q$ is $X^{q+1}$. As the basis sets $X^i$ are nested and symmetric, it is fairly easy to see that $H(q,d)$ is union of fully symmetric sets and can be explicitly written so:
\begin{align*}
H(q,d) &= \bigcup_{\substack{\abs[0]{\alpha} = d + q \\ \alpha_i \geq \alpha_{i+1}}} \, \bigcup_{\mq \in \Pi_d} \big( X^{\alpha(q_1)} \times \cdots \times X^{\alpha(q_d)} \big) \\
&= \bigcup_{\substack{\abs[0]{\alpha} = d + q \\ \alpha_i \geq \alpha_{i+1}}} \, \bigcup_{\mq \in \Pi_d} \bigcup_{\ms \in S_d} \, \bigcup_{\substack{ \mlambda \in \R^d \\ \lambda_j \in X^{\alpha(q_j)} \\ \lambda_j \geq 0}} \big\{ (s_1 \lambda_1, \ldots, s_d \lambda_d) \big\} \\
&= \bigcup_{\substack{\abs[0]{\alpha} = d + q \\ \alpha_i \geq \alpha_{i+1}}} \Set[\big]{[\lambda_1, \ldots, \lambda_d] }{\lambda_j \in X^{\alpha_j} \text{ and } \lambda_j \geq 0 \text{ for } j = 1,\ldots,d },
\end{align*}
where the restriction $\alpha_1 \geq \alpha_{i+1}$ eliminates a large number of permutations that would be otherwise duplicated when generating fully symmetric sets. That is, \Cref{thm:main} applies to sparse grids. We call the resulting kernel quadrature rules the \emph{sparse grid kernel quadrature rules}.

We are left with selection of the nested point sets $X^i$. In polynomial-based sparse grid quadrature rules these sets come coupled with univariate quadrature rules whose weights are used to construct the final sparse grid weights but we are under no such restrictions. An obvious idea for selecting $X^i$ would be to sequentially minimise the worst-case error in one dimension---provided that the kernel is one for which this makes sense, for example any of the three examples of fully symmetric kernels given in \Cref{sec:fskernels}. Discussion on different sequential kernel quadrature methods (usually known as \emph{sequential Bayesian quadratures}) can be found in for example~\cite{CookClayton1998,HuszarDuvenaud2012,Gunter2014,BriolOatesGirolamiOsborne2015}. A different approach appears in~\cite{Oettershagen2017}.

However, owing to difficulties in setting the kernel length-scale, we do not employ this selection scheme. Instead, we use (i) the Clenshaw--Curtis point sets, rather standard in sparse grid literature, for $\Omega = [-1,1]^d$ and the uniform measure and (ii) nested sets formed out of Gauss--Hermite nodes in the Gaussian case:
\begin{enumerate}
\item[(i)] For $i > 1$ and $m_i = 2^{i-1} + 1$, the nested Clenshaw--Curtis sets are
\begin{equation*}
X^i = \{x_1^i,\ldots,x_{m_i}^i\} \quad \text{with} \quad x_j^i = -\cos \bigg(\frac{\pi(j-1)}{m_i - 1}\bigg) \in [-1,1].
\end{equation*}
The points $x_j^i$ are the roots and extrema of Chebyshev polynomials. The corresponding sparse grid kernel quadrature rule is called \emph{Clenshaw--Curtis sparse grid kernel quadrature} (CCSGKQ). Numerical results for this kernel quadrature are given in \Cref{sec:exapriori} and convergence for sufficiently smooth functions is the topic of \Cref{thm:SGconv}.
\item[(ii)] In the \emph{Gauss--Hermite sparse grid kernel quadrature} (GHSGKQ) we use the classical Gauss--Hermite nodes that are the roots of the Hermite polynomials
\begin{equation*}
\her_p(x) = (-1)^p \exp\big(x^2/2\big) \od[p]{}{x} \exp\big(\!-x^2/2\big).
\end{equation*}
Given a level $q$, we generate the $2q+1$ symmetric roots of $\her_{2q+1}$ and for \sloppy{${i=1,\ldots,q+1}$} select
\begin{equation*}
X^i = \text{\,the $2i - 1$ smallest roots by absolute value.}
\end{equation*}
The number of nodes, in terms of the level $q$, grows significantly slower than with Clenshaw--Curtis sparse grids. A numerical experiment involving a financial problem is given in \Cref{sec:zcb}. As is usual for quadrature rules on the whole of $\R^d$, there are no theoretical convergence guarantees for the GHSGKQ. Because $H(q,d)$ is not a subset of $H(q+1,d)$ for the Gauss--Hermite points, these grids are only suitable for cases where the number of nodes that can be used is determined beforehand based on for example the computational budget available. Note that these grids are completely different from several other sparse grids in the literature that use nodes of Gaussian quadrature rules.
\end{enumerate}

\begin{figure}[t]
\centering
  \begin{subfigure}[b]{0.44\textwidth}
  \centering
  \includegraphics{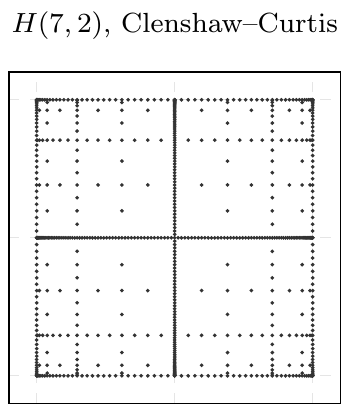}
  \end{subfigure}
  \begin{subfigure}[b]{0.44\textwidth}
  \centering
  \includegraphics{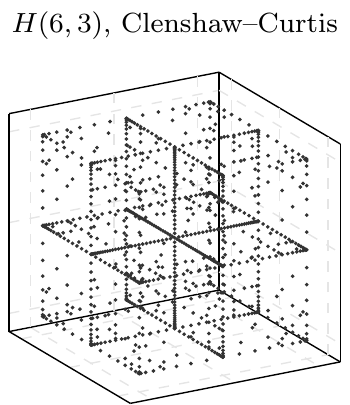}
  \end{subfigure}
  \begin{subfigure}[b]{0.44\textwidth}
  \centering
  \vspace{0.3cm}
  \includegraphics{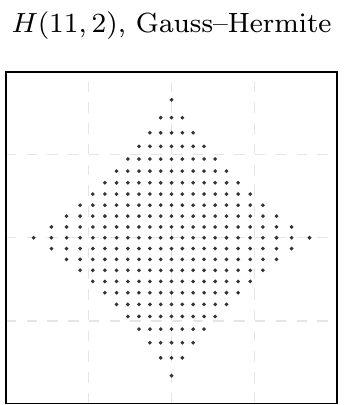}
  \end{subfigure}
  \begin{subfigure}[b]{0.44\textwidth}
  \centering
  \includegraphics{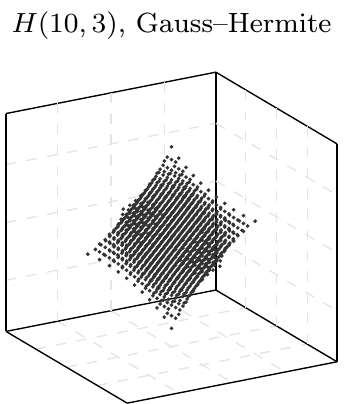}
  \end{subfigure}
  \caption{Examples of sparse grids with Clenshaw--Curtis and Gauss--Hermite nodes. The numbers of nodes in the sparse grids are 705 (upper left), 1,073 (upper right), 265 (lower left), and 1,561 (lower right). Compare these to the cardinalities of the corresponding full grids that are $129^2=16{,}641$; $65^3=274{,}625$; $23^2 = 529$; and $21^3=9{,}261$.}\label{fig:sparsegrids}
\end{figure}

Four sparse grids based on these two point sequences are depicted in \Cref{fig:sparsegrids}. There is a large array of other possibilities available in the literature. For example, Gerstner and Griebel~\cite{GerstnerGriebel1998} use Gauss--Patterson nodes and Genz and Keister~\cite{GenzKeister1996} have developed a nested version of the Gauss--Hermite rule. The rule (ii) can also be trivially extended for other integration domains and measures if a different sequence of orthogonal polynomials is used (e.g.\ Legendre or Chebyshev polynomials on $[-1,1]$).

\subsection{Worst-case error minimisation with respect to the generators}\label{sec:wcemin}

The third methodology for choosing the fully symmetric sets is that of principled worst-case error minimisation. Suppose that one, based on for example the number of nodes desired, fixes a number of generators of a fully symmetric kernel quadrature rule to zero or sets some equality constraints. Then the worst-case error $e(Q_k)$ of the kernel quadrature rule can be minimised with respect to the generator vectors obeying these constraints. Especially in higher dimensions, this is a task vastly simpler than trying to minimise the error over a node set of unconstrained geometry. Optimal kernel quadrature rules under certain structural constraints have been previously experimented with at least by O'Hagan~\cite{OHagan1991,OHagan1992}.

As a simplistic and somewhat arbitrary example, suppose that one desires a good fully symmetric kernel quadrature rule having about 80 nodes in $\Omega \subset \R^3$. A rule of the form
\begin{equation}\label{eq:fskq-optimisation}
Q_k^{\mlambda}(f) = w_1 f[0,0,0] + w_2 f[\lambda_1,\lambda_2,\lambda_3] + w_3f[\lambda_3,\lambda_3,\lambda_3] + w_4 f[\lambda_4,\lambda_4,\lambda_5]
\end{equation}
has $1+48+8+24 = 81$ nodes if the generators $\mlambda = (\lambda_1,\ldots,\lambda_5)$ are distinct and non-zero. Optimal generators $\mlambda^* = (\lambda^*_1,\lambda_2^*,\lambda_3^*,\lambda_4^*,\lambda_5^*)$ in the sense of minimal worst-case error would then be
\begin{equation*}
\mlambda^* = \argmin_{\mlambda \in \Omega, \lambda_i > 0} e(Q_k^{\mlambda}).
\end{equation*}
That is, $Q_k^{\mlambda^*}$ has the smallest WCE among all rules of the form~\eqref{eq:fskq-optimisation}. The optimal generators cannot be in general solved analytically nor does this minimisation problem appear to be convex.

In some very simple cases minimisation is trivial. Consider rules of the form
\begin{equation*}
Q_k^\lambda (f) = w_1 f[0,\ldots,0] + w_2 f[\lambda,0,\ldots,0]
\end{equation*}
in $\Omega \subset \R^d$. If the kernel is Gaussian with length-scale $\ell$ and $\mu$ the standard Gaussian measure (see \Cref{sec:closedformkmean}), the task of finding the optimal generator $\lambda^*$ reduces to
\begin{equation*}
\lambda^* = \argmin_{\lambda > 0} e(Q_k^\lambda) = \argmax_{\lambda > 0}\Bigg[\frac{\neper^{-\lambda^2/2(1+\ell^2)} - \neper^{-\lambda^2/2\ell^2}}{1 - \neper^{-\lambda^2/\ell^2}}\Bigg].
\end{equation*}
That is, the optimal generator is dimension-independent and can be easily computed. However, with increasing dimension and constant length-scale this results to a negative weight for the origin which often impairs numerical stability. It is somewhat questionable if such a rule is actually ``good'' (removing the origin of course yields positive weights).

We do not attempt to construct efficient fully symmetric rules using the technique described above in this article. The topic, alongside with node selection for sparse grids via sequential WCE minimisation briefly discussed in \Cref{sec:smolyak}, is left for future research.

\subsection{Convergence analysis}\label{sec:convergence}

In this section we provide convergence theorems for the fully symmetric kernel Monte Carlo and the Clenshaw--Curtis sparse grid kernel quadrature. The theorems are straightforward corollaries of some well-known results in the literature. For stating the results, we need to introduce the following three standard function classes. We assume that $\Omega = [-a,a]^d$ and $a < \infty$ in this section. The general principle on the convergence results for kernel quadrature is that the rates obtained are at least as good as those for any other method using the same nodes if the integrand belongs to the RKHS induced by the kernel. This should be quite clear from the definition of kernel quadrature.

With $\alpha \in \N^d$ a multi-index and $f \colon \Omega \to \R$ a sufficiently smooth function, the derivative operator is $\difop^\alpha f = \partial^{\abs[0]{\alpha}}f / \partial x_1^{\alpha_1} \cdots \partial x_d^{\alpha_d}$. By $\alpha \leq r$ we mean that \sloppy{${\alpha_1,\ldots,\alpha_d \leq r}$}. The function classes we need are
\begin{enumerate}
\item[(i)] The Sobolev space $\sobolev_2^r$ is a Hilbert space defined as
\begin{equation*}
\sobolev_2^r \coloneqq \Set[\big]{ f \in L^2(\mu) }{ \difop^\alpha f \in L^2(\mu) \text{ exists for all } \abs[0]{\alpha} \leq r}
\end{equation*}
with the norm $\norm[0]{f}_{W^r_2} = \sum_{\abs[0]{\alpha} \leq r} \norm[0]{\difop^\alpha f}_{L^2(\mu)}$.
\item[(ii)] The class $C^r$ is the class of functions that have bounded derivatives:
\begin{equation*}
C^r \coloneqq \Set[\big]{f \colon \Omega \to \R}{\norm[0]{\difop^\alpha f}_\infty < \infty \text{ for all } \abs[0]{\alpha} \leq r}.
\end{equation*}
This space is equipped with the norm $\norm[0]{f}_{C^r} = \max\Set{\norm[0]{\difop^\alpha f}_\infty}{\abs[0]{\alpha} \leq r}$.
\item[(iii)] The class $F^r$ is the class of functions that have bounded mixed derivatives:
\begin{equation*}
F^r \coloneqq \Set[\big]{f \colon \Omega \to \R}{\norm[0]{\difop^\alpha f}_\infty < \infty \text{ for all } \alpha \leq r}.
\end{equation*}
This space is equipped with the norm \sloppy{${\norm[0]{f}_{F^r} = \max\Set{\norm[0]{\difop^\alpha f}_\infty}{\alpha \leq r}}$}.
\end{enumerate}
For relations of the above function classes to reproducing kernel Hilbert spaces induced by different kernels, see for example~\cite[Chapter 10]{Wendland2005}. Two norms $\norm[0]{\cdot}_1$ and $\norm[0]{\cdot}_2$ of an arbitrary vector space are \emph{norm-equivalent} if there are positive constants $C_1$ and $C_2$ such that $C_1\norm[0]{x}_1 \leq \norm[0]{x}_2 \leq C_2\norm[0]{x}_1$ for all elements $x$ of the vector space. Recall from \Cref{sec:kernelquad} that $\rkhs$ is the RKHS induced by the kernel~$k$. Then $\rkhs$ is norm-equivalent to the Sobolev space $\sobolev_2^r$ if $r > d/2$ and the Fourier transform $\mathcal{F}(\omega)$ of the kernel $k$ decays at the rate $(1 + \norm[0]{\omega}^2)^{-r}$. This holds if the kernel is for example of the Matérn class with $\nu \geq r - 1/2$. In a similar manner, $\rkhs$ is norm-equivalent to $F^r$ if the kernel is a product of one-dimensional Matérn kernels.

The following convergence theorems are simple consequences of results available in the literature. In these theorems $Q_{k,n}$ stands for an $n$-point kernel quadrature rule with its type specified by a superscript. Extensions to the misspecified setting analysed in~\cite{KanagawaSriperumbudurFukumizu2016,KanagawaSriperumbudurFukumizu2017} may be possible. 

\begin{theorem}[Convergence of  FSKMC]\label{thm:RGconv} Let $\Omega = [-a, a]^d$ with $a < \infty$. If $\rkhs$ is norm-equivalent to the Sobolev space $\sobolev_2^r$ with $r > d/2$, then
\begin{equation}\label{eq:MCrate}
\expec \big[e(Q_{k,n}^{\textsm{FSKMC}})\big] = \bigO\big(n^{-r/d + \epsilon}\big)
\end{equation}
for any $\epsilon > 0$. The expectation above is with respect to the joint distribution of the random generator vectors.
\end{theorem}
\begin{proof} The worst-case error is decreasing in the number of nodes so we know that $e(Q_{k,n}^{\textsm{FSKMC}}) \leq e(Q_{k,n}^{\textsm{GEN}})$ where the rule $Q_{k,n}^{\textsm{GEN}}$ uses only the $J$ generator vectors as its nodes. The rate~\eqref{eq:MCrate} is realised by the regular kernel Monte Carlo~\mbox{\cite[Theorem 1]{BriolOatesGirolamiOsborneSejdinovic2016}} under the norm-equivalence assumption. Therefore,
\begin{equation*}
\expec \big[e(Q_{k,n}^{\textsm{FSKMC}})\big] \leq \expec \big[e(Q_{k,n}^{\textsm{GEN}})\big] = \bigO\big(J^{-r/d + \epsilon}\big) = \bigO\big(n^{-r/d + \epsilon}\big)
\end{equation*}
because there is a dimension-dependent upper bound $2^d d!$ (see \Cref{eq:FSScardinality}) for the number of nodes one fully symmetric set can contain.
\end{proof}

It is clear that the above rate is extremely crude with respect to $d$ as the dimension also enters through the multiplicative factor $2^d d!$. It is likely that this factor can be eliminated or diminished with more careful analysis.
 
\begin{theorem}[Convergence of CCSGKQ]\label{thm:SGconv} Let $\Omega = [-a, a]^d$ with $a < \infty$ and $\mu$ be the uniform measure on $\Omega$. If $\rkhs$ is norm-equivalent to $C^r$, then 
\begin{equation}\label{eq:Cbound}
e(Q_{k,n}^\textsm{CCSGKQ}) = \bigO\big(n^{-r/d} (\log n)^{(d-1)(r/d+1)}\big).
\end{equation}
If $\rkhs$ is norm-equivalent to $F^r$, then
\begin{equation}\label{eq:Fbound}
e(Q_{k,n}^\textsm{CCSGKQ}) = \bigO\big(n^{-r} (\log n)^{(d-1)(r+1)}\big).
\end{equation}
\end{theorem}
\begin{proof} The rates~\eqref{eq:Cbound} and~\eqref{eq:Fbound} hold for the standard Clenshaw--Curtis sparse grid quadrature~\cite{NovakRitter1996,NovakRitter1997} if the worst-case error~\eqref{eq:RKHSerror} is over the unit balls of $C^r$ and $F^r$, respectively. Because kernel quadrature rules have minimal worst-case errors in the induced RKHS among all quadrature rules with fixed nodes, the convergence rates follow from the assumptions of norm-equivalence.
\end{proof}

\section{Numerical examples and computational aspects}\label{sec:num}

This section contains three numerical examples for the fully symmetric kernel Monte Carlo, the Clenshaw--Curtis sparse grid kernel quadrature and the (modified) Gauss--Hermite sparse grid kernel quadrature as well as discussion on some computational aspects. The examples and algorithms, implemented in MATLAB, are available at \texttt{https://github.com/tskarvone/fskq}. Numerous classical sparse grid quadrature methods for MATLAB are implemented in the Sparse Grid Interpolation Toolbox~\cite{spinterp-art}. Parts of our code make use of this toolbox.

We emphasise the examples are not meant to demonstrate superiority of fully symmetric kernel quadrature to other numerical integration methods. Comparisons to other methods are merely to show that fully symmetric kernel quadratures can achieve roughly comparable accuracy. Rather, we aim to show that fully symmetric sets make it possible to apply kernel quadrature rules to large-scale and high-dimensional situations that have been out of the scope of these quadrature rules before.

\subsection{Choosing the length-scale parameter}\label{sec:lengthscale}

Accuracy of any approximation based on a stationary kernel is heavily dependent on the length-scale parameter $\ell > 0$ whose effect is via $k_\ell(\mx-\mx') = k((\mx-\mx')/\ell)$, see for example the Gaussian kernel~\eqref{eq:gausskernel}. Choosing in some sense the best value of this parameter efficiently is an important topic of research both in scattered data approximation literature~\cite{Rippa1999,FasshauerZhang2007} and statistics and machine learning~\cite[Chapter 5]{RasmussenWilliams2006}. See also~\cite[Section 4.1]{BriolOatesGirolamiOsborneSejdinovic2016} for discussion in the context of kernel quadrature.

Unfortunately, we have not been able to come up with a way to exploit the fully symmetric structure of the node set in any of the existing parameter fitting methods, such as marginal likelihood maximisation or cross-validation. Consequently, in large-scale applications that go beyond the limits of naive methods based on inverting the kernel matrix one has to resort to ad hoc techniques to fit the length-scale. In the examples below we either use few enough nodes that naive computations are possible (\Cref{sec:exrand}), integrate a function whose length-scale is known beforehand (\Cref{sec:exapriori}), or set the length-scale somewhat heuristically (\Cref{sec:zcb}). We recognise that the lack of a principled method for choosing the length-scale is a significant shortcoming and hope to fix this in the future.

When the length-scale is changed, the interpretation of the worst-case error as an indicator of accuracy of the quadrature rule is confounded because the RKHS norm $\norm[0]{\cdot}_\rkhs$ depends on the length-scale. This occurs in \Cref{sec:exrand,sec:zcb}. Nevertheless, if one follows the paradigm presented in \Cref{sec:probnum} the WCE still carries a meaningful probabilistic interpretation as the integral posterior standard deviation (STD). As such, it is plotted in all the examples. However, one should not draw too many conclusions from these plots as we have not made any effort to fit the kernel scale parameter (i.e.\ the constant multiplier of the kernel).

\subsection{Closed-form kernel means}\label{sec:closedformkmean}

In kernel quadrature, one needs to be able to evaluate the kernel mean $\kmean(\mx_i) = \int_\Omega k(\mx_i,\mx)\dif\mu(\mx)$ at the nodes $\mx_i$. A number of kernel-measure pairs that yield tractable kernel means are tabulated in~\cite{BriolOatesGirolamiOsborneSejdinovic2016}. It is also possible to evaluate the kernel mean numerically~\cite{SommarivaVianello2006a}. In fact, when fully symmetric sets are used, numerical evaluation may be quite viable as the kernel mean needs to be evaluated only at each generator vector instead of each node.

All our examples use the standard Gaussian kernel
\begin{equation}\label{eq:gausskernel}
k(\mx,\mx') = \exp\bigg(\!-\frac{\norm[0]{\mx-\mx'}^2}{2\ell^2}\bigg)
\end{equation}
with length-scale $\ell > 0$ and unit scale parameter (this parameter only affects magnitude of the worst-case error). The integration domain $\Omega$ and measure $\mu$ are either (i) the whole of $\R^d$ and the standard Gaussian measure with the density $\phi(\mx) = (2\pi)^{-d/2}\exp(-\norm[0]{\mx}^2/2)$ or (ii) the hypercube $[-1, 1]^d$ and the (normalising) uniform measure. In the former case the kernel mean and its integral (needed for computing the WCE using \Cref{eq:wce}) are
\begin{align*}
\kmean(\mx) &= \int_{\R^d} k(\mx,\mx') \phi(\mx') \dif \mx' = \bigg(\frac{\ell^2}{1+\ell^2}\bigg)^{d/2} \exp\bigg(\! - \frac{\norm[0]{\mx}^2}{2(1+\ell^2)}\bigg), \\
\mu(\kmean) &= \int_{\R^d} \kmean(\mx) \phi(\mx)\dif \mx = \bigg(\frac{\ell^2}{2+\ell^2}\bigg)^{d/2}
\end{align*}
and in the latter
\begin{align*}
\kmean(\mx) &= 2^{-d}\int_{[-1,1]^d} k(\mx,\mx') \dif \mx' = \bigg(\frac{\pi\ell^2}{8}\bigg)^{d/2} \prod_{i=1}^d \Bigg[\erf\bigg(\frac{x_i + 1}{\ell\sqrt{2}}\bigg) - \erf\bigg(\frac{x_i - 1}{\ell\sqrt{2}}\bigg)\Bigg], \\
\mu(\kmean) &= 2^{-d}\int_{[-1,1]^d} \kmean(\mx) \dif \mx = \bigg(\frac{\pi\ell^2}{8}\bigg)^{d/2} \Bigg(\sqrt{\frac{2\ell^2}{\pi}}\big(\neper^{-2/\ell^2} - 1\big) + 2 \erf\big(\sqrt{2}/\ell\big)\Bigg)^d,
\end{align*}
where $\erf(x) = \pi^{-1/2}\int_{-x}^x \exp(-t^2)\dif t$ is the standard error function.

\subsection{Example 1: Random generators}\label{sec:exrand}

\begin{figure}[t]
\centering
  \begin{subfigure}[b]{\textwidth}
  \centering
  \includegraphics{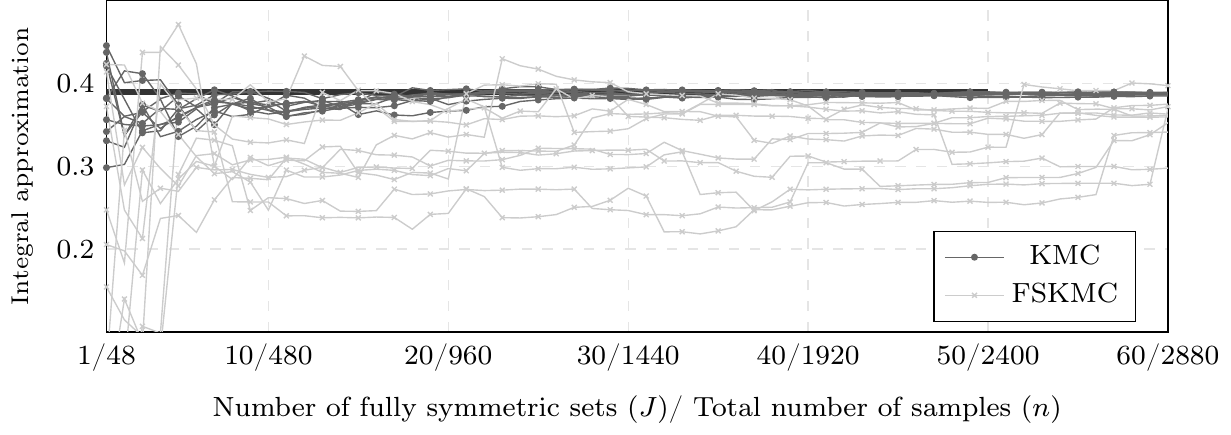}
  \end{subfigure}
  \begin{subfigure}[b]{0.48\textwidth}
  \centering
  \includegraphics{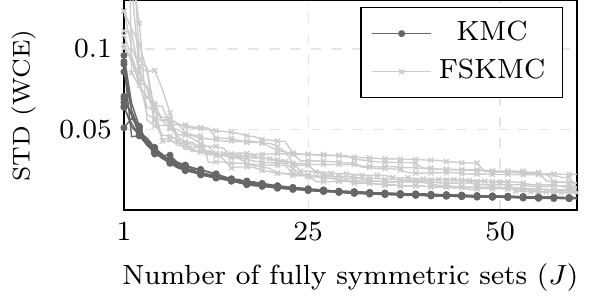}
  \end{subfigure}
  \hspace{0.1cm}
  \begin{subfigure}[b]{0.48\textwidth}
  \centering
  \includegraphics{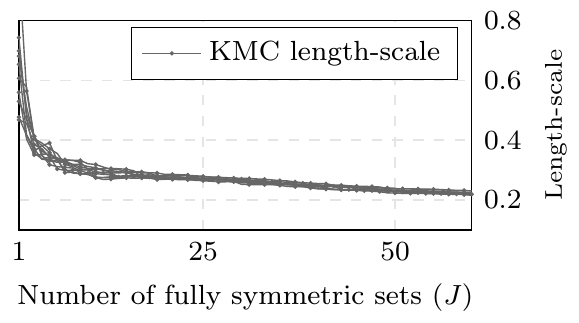}
  \end{subfigure}
  \caption{Numerical integration of the function~\eqref{eq:ex1func} with respect to the standard normal distribution. The upper figure shows integral approximations by the kernel Monte Carlo (KMC) and the fully symmetric kernel Monte Carlo (FSKMC) as a function of the number $J$ of fully symmetric sets and the total number $n$ of Monte Carlo samples for ten realisations. Lower figures display the worst-case errors (standard deviations) and the length-scales fitted. Each fully symmetric set contains 48 nodes (see \cref{table:RGsize}). The underlying black line is the value of the integral that is approximately $0.389$. The generator vectors for the FSKMC have been generated independently of the KMC samples. Both methods use the same length-scale that has been fit using the MC samples of the KMC.}\label{fig:randgencomp}
\end{figure}

Our first example is just a proof of concept to demonstrate that the fully symmetric kernel Monte Carlo (FSKMC) from \Cref{sec:fskmcq} indeed works (though not necessarily that well). We try numerically integrating the non-radial function
\begin{equation}\label{eq:ex1func}
f(\mx) = \exp\Big( \sin(5\norm[0]{\mx})^2 - (x_1^2 + 0.5x_2^2 + 2x_3^4) \Big)
\end{equation}
over $\R^3$ and with respect to the standard normal distribution. Results for the kernel Monte Carlo (KMC)~\cite{RasmussenGhahramani2002,BriolOatesGirolamiOsborneSejdinovic2016}, where the nodes to be used in kernel quadrature are drawn randomly, and the fully symmetric kernel Monte Carlo are presented in \Cref{fig:randgencomp}.

For both the KMC and FSKMC, the kernel length-scale was fit by the method of maximum likelihood (see~\cite[Chapter 5]{RasmussenWilliams2006}) using the MC samples of KMC. We have also experimented with fitting the FSKMC length-scale using the randomly generated fully symmetric sets, but in this case the fitted length-scale was markedly larger and the integral approximations much worse.

It is clear that the FSKMC fares worse. However, the FSKMC has a tremendous advantage in computational scalability in the number of nodes. In general, when the number of nodes exceeds some tens of thousands, kernel quadrature methods based on naively solving the weights from the linear system~\eqref{eq:KQweights}, such as the KMC, become infeasible due to their the cubic time and quadratic space complexity. In contrast, fully symmetric kernel quadratures such as FSKMC remain feasible: only $Jn$ (recall that $J \leq n$ is the number of fully symmetric sets) kernel evaluations and solving a linear system of $J$ equations, as opposed to $n^2$ and $n$, respectively, of naive methods, are required. For instance, using FSKMC with 1,000 fully symmetric sets (i.e.\ 48,000 nodes) in this example would require 48,000,000 kernel evaluations and solving a linear system of 1,000 equations, neither of which is a computational challenge, while the KMC weights for 48,000 nodes cannot be computed on a standard computer.

The difference becomes even more pronounced in higher dimensions where fully symmetric sets contain significantly more points. The next two examples demonstrate the superior scalability of fully symmetric kernel quadratures.

\begin{figure}[t]
\centering
  \begin{subfigure}[b]{\textwidth}
  \centering
  \includegraphics{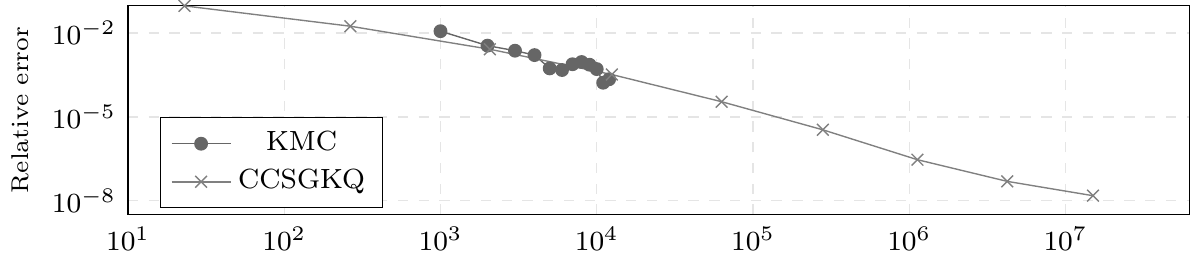}
  \end{subfigure}
  \begin{subfigure}[b]{\textwidth}
  \centering
  \includegraphics{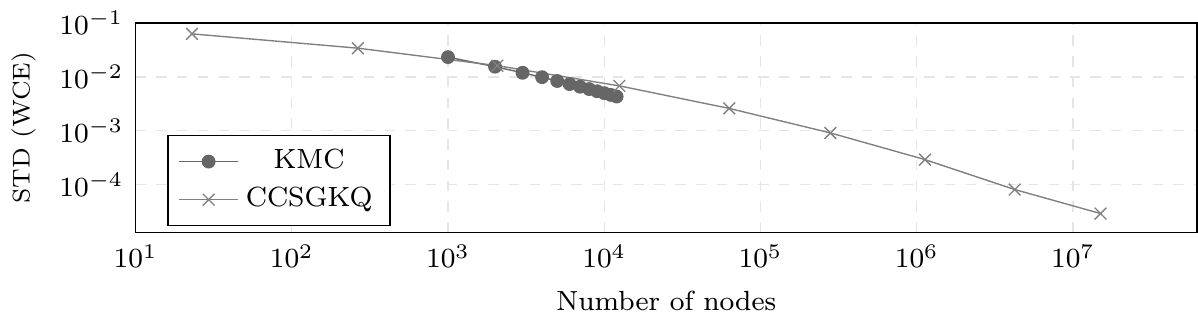}
  \end{subfigure}
  \caption{Relative error $\abs[0]{\mu(f) - Q_k(f)}/\mu(f)$ (upper) and the worst-case error (lower; standard deviation) for integration of the function~\eqref{eq:apriorifunc} on the 11-dimensional hypercube with the kernel Monte Carlo quadrature (KMC) and the Clenshaw--Curtis sparse grid kernel quadrature (CCSGKQ). The number of nodes for the KMC varied from 1,000 to 12,000 (with increments of 1,000) and CCSGKQ used levels from 1 to 9. These corresponded to 23; 265; 2,069; 12,497; 63,097; 280,017; 1,129,569; 4,236,673; and 15,005,761 nodes and 2, 4, 8, 17, 36, 79, 172, 379, and 832 fully symmetric sets.}\label{fig:apriori}
\end{figure}

\subsection{Example 2: A priori known length-scale}\label{sec:exapriori}

This simple example demonstrates that sparse grid kernel quadrature based on fully symmetric sets is numerically stable, consistent and works well for an extremely large number of nodes.

We work in the domain $\Omega = [-1,1]^d$, $d=11$, equipped with the normalising uniform measure. The integrand is
\begin{equation}\label{eq:apriorifunc}
f(\mx) = \exp\bigg(\!-\frac{\norm[0]{\mx-\mx_f}^2}{2\ell_f^2}\bigg)
\end{equation}
with $\ell_f = 0.8$ and $\mx_f$ is a vector of $11$ evenly spaced points on the interval $[0.2, 0.5]$ (with the end points included). The integral we seek to approximate is
\begin{equation*}
2^{-d}\int_{[-1,1]^{d}} f(\mx) \dif \mx = \bigg(\frac{\pi\ell_f^2}{8}\bigg)^{d/2} \prod_{i=1}^{d} \Bigg[\erf\bigg(\frac{x_{f,i} + 1}{\ell_f\sqrt{2}}\bigg) - \erf\bigg(\frac{x_{f,i} - 1}{\ell_f\sqrt{2}}\bigg)\Bigg] \approx 0.0392.
\end{equation*}
We use the Gaussian kernel with $\ell = \ell_f = 0.8$ and the Clenshaw--Curtis sparse grid kernel quadrature (CCSGKQ). Results for the relative error $\abs[0]{\mu(f) - Q_k(f)}/\mu(f)$ and the kernel worst-case error (or standard deviation) are shown in \Cref{fig:apriori} for the levels $q=1,\ldots,9$, last of them corresponding to the total of 15,005,761 nodes. \Cref{table:tocs} contains a breakdown of the computational times required. We also display results for the kernel Monte Carlo using up to 12,000 nodes. For this many nodes, the time taken by the CCSGKQ is negligible while the KMC is noticeably slowing down.

\begin{table}[t]
\begin{center}
\captionof{table}{Computational times in seconds for the KMC (left) and CCSGKQ (right) in Example 2. The columns indicate the time taken by kernel evaluations (kernel), computing the weights from a linear system of equations (weights), and constructing the fully symmetric sets (FSS). The MATLAB code was run on a laptop sporting Intel Core i5-6300 2.40 GHz processor and 8 GB of RAM.}\label{table:tocs}
\small
\caption*{\small{\textbf{Computational times (seconds) for KMC / CCSGKQ}}}\vspace{-0.1cm}
\begin{tabular}[t!]{c|c c}
\arrayrulecolor{gray4}
$n$ & kernel & weights \\
\Xhline{1pt}
1k & 0.08 & 0.01 \\
2k & 0.15 & 0.07 \\
3k & 0.32 & 0.20 \\
4k & 0.54 & 0.47 \\
5k & 0.79 & 0.85 \\
6k & 1.17 & 1.42 \\ 
7k & 1.49 & 2.21 \\ 
8k & 1.92 & 3.20 \\ 
9k & 2.43 & 4.47 \\ 
10k & 2.98 & 6.03 \\ 
11k & 3.62 & 8.21 \\ 
12k & 4.45 & 10.61 \\
\end{tabular} \hspace{1cm}
\begin{tabular}[t!]{c|c c c}
\arrayrulecolor{gray4}
$J$ / $n$ & kernel & weights & FSS \\
\Xhline{1pt}
2 / 23 & $< 0.01$ & $< 0.001$ & 0.07 \\
4 / 265 & $< 0.01$ & $< 0.001$ &  0.07 \\
8 / 2k & $< 0.01$ & $< 0.001$ & 0.04 \\ 
17 / 12k & 0.02 & $< 0.001$ & 0.04 \\
36 / 63k & 0.14 & $< 0.001$ & 0.10 \\ 
79 / 280k & 1.26 & 0.003 & 0.27 \\
172 / 1.1m & 10.91 & 0.004 & 0.82 \\
379 / 4.2m & 90.41 & 0.004 & 2.63 \\
832 / 15m & 760.00 & 0.072 & 8.61 \\
\multicolumn{3}{c}{} \\
\multicolumn{3}{c}{} \\
\multicolumn{3}{c}{} \\
\end{tabular}
\end{center}
\end{table}

It is seen that for a similar numbers of nodes the two methods are roughly equivalent. When the level, and consequently the number of nodes, increases, the CCSGKQ becomes more accurate which shows that the nodes are selected well enough and that the weights are being computed correctly. For the highest levels 8 and 9, the sparse grids consisted of 379 and 832 fully symmetric sets or 4,236,673 and 15,005,761 nodes, resulting in \sloppy{${379 \times 4{,}236{,}673 = 1{,}605{,}699{,}067}$} and $832 \times 15{,}005{,}761 = 12{,}484{,}793{,}152$ kernel evaluations needed to compute the weights. It is not possible to compute the KMC weights for this many nodes.

\subsection{Example 3: Zero coupon bonds}\label{sec:zcb}

This example demonstrates that fully symmetric kernel quadrature rules are also feasible in high dimensions. We use the toy example from~\cite{NinomiyaTezuka1996} (see also~\cite[Section 6.1]{Holtz2011}) that is concerned with pricing zero coupon bonds through simulation of a discretised stochastic differential equation model. The model is convenient for our purposes as there is a closed-form solution that serves as a baseline and finer discretisations correspond to higher integration dimensions.

Consider the stochastic differential equation (going by the name Vasicek model)
\begin{equation*}
\dif r(t) = \kappa\big(\theta - r(t)\big)\dif t + \sigma \dif W(t),
\end{equation*}
where $W(t)$ is the standard Brownian motion and $\kappa$, $\theta$, and $\sigma$ are positive parameters. We want to solve this SDE at the time $t=T$. The Euler--Maruyama discretisation with the uniform step size $\Delta t = T/d$ is
\begin{equation*}
r_k = r_{k-1} + \kappa(\theta - r_{k-1})\Delta t + \sigma x_k, \quad k=1,\ldots,d,
\end{equation*}
where $x_k \sim \gauss(0, \Delta t)$ are independent and $r_0$ is a free parameter. The quantity we are interested in is the Gaussian integral
\begin{equation}\label{eq:zcbintegral}
\begin{split}
P(0,T) &\coloneqq \expec\bigg[\exp\bigg(-\Delta t \sum_{k=0}^{d-1} r_k \bigg)\bigg] \\
&= \exp(-\Delta t r_0) \int_{\R^{d-1}} \exp\Big(-\Delta t f\big(\sqrt{\Delta t}\mx\big) \Big) \phi(\mx) \dif \mx,
\end{split}
\end{equation}
where $f(\mx) = \sum_{k=1}^{d-1} r_k$. This integral admits a closed-form solution
\begin{equation}\label{eq:zcbtrue}
P(0,T) = \exp\bigg(-\frac{(\gamma + \beta_d r_0) T}{d}\bigg)
\end{equation}
with $\beta_k = \sum_{i=1}^k (1 - \kappa\Delta t)^{j-1}$ and $\gamma = \sum_{k=1}^{d-1} (\beta_k \kappa\theta\Delta t - (\beta_k\sigma\Delta t)^2/2)$.
As can be seen, the number $d$ of discretisation steps controls the integration dimension that is $d-1$.

In the integration experiment, we set
\begin{equation*}
\kappa = 0.1817303, \quad \theta = 0.0825398957, \quad \sigma = 0.0125901, \quad r_0 = 0.021673, \quad T = 5.
\end{equation*}
These values are equal to those used in~\cite{NinomiyaTezuka1996,Holtz2011}. We consider numerical integration of~\eqref{eq:zcbintegral} for $d = 10,\ldots,300$ using the Gauss--Hermite sparse grid kernel quadrature (GHSGKQ) with $q=2$. We use the Gaussian kernel with the somewhat heuristic choice $\ell = d$ of the length-scale. The central node (i.e.\ the origin) tended to have a fairly large negative weight so it was removed to improve numerical stability. Results for the relative error and the worst-case error (standard deviation) are depicted in \Cref{fig:zcb}. For comparison, we have also included a Monte Carlo estimate (KMC is feasible only for dimensions somewhat less than 100 so it was excluded).

The results show that the GHSGKQ is able to maintain an accuracy that is generally better than that of the standard MC. This indicates that fully symmetric kernel quadratures have potential also in very high dimensions. Note the dimension-adaptive methods used in~\cite{Holtz2011} would be more accurate in this example. It is probable that fully symmetric kernel quadratures could be combined with these methods.

\begin{figure}[t]
  \centering
  \includegraphics{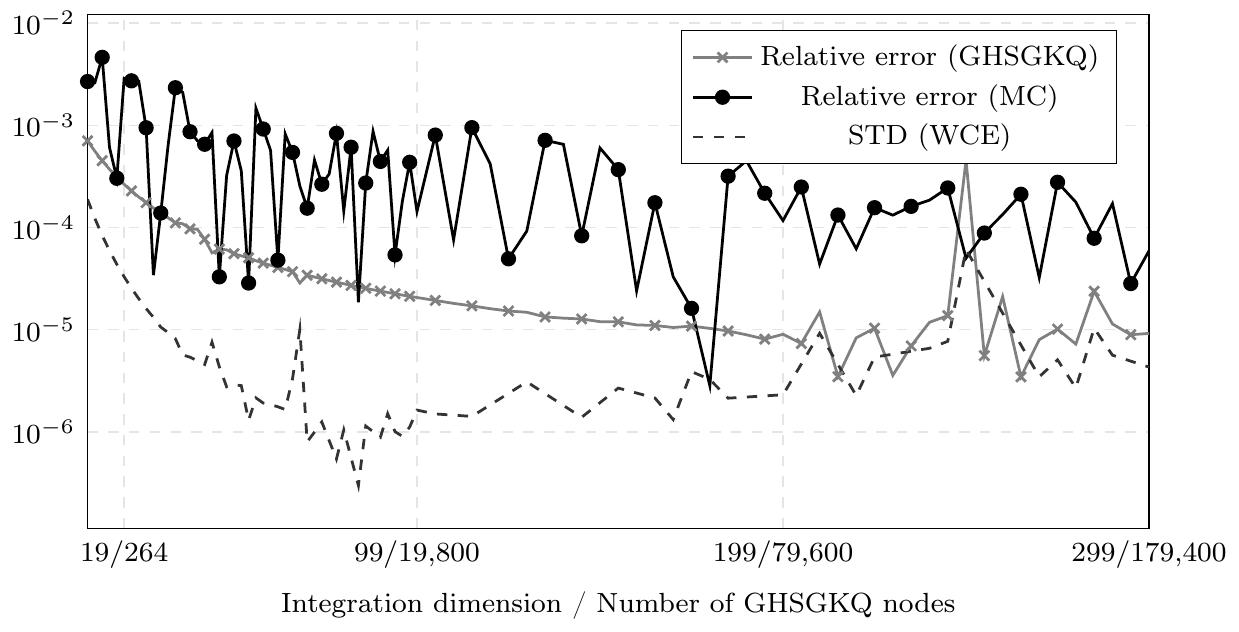}
  \caption{Relative error and the worst-case error (standard deviation) of the Gauss--Hermite sparse grid kernel quadrature (GHSGKQ) for the zero coupon bond setup of \Cref{sec:zcb}. Value of the integral~\eqref{eq:zcbintegral}, as computed from \Cref{eq:zcbtrue}, is between 0.81 and 0.815 for all dimensions. Also depicted is an integral estimate by the standard Monte Carlo using the same number of points as the GHSGKQ in each dimension.}\label{fig:zcb}
\end{figure}

\section{Conclusions and discussion}

We introduced fully symmetric kernel quadrature rules and showed that their weights can be computed exactly with a very simple algorithm under some assumptions on the integration domain and measure and the kernel. We also proposed using sparse grids in conjunction with this algorithm and provided some simple theoretical convergence analysis for Clenshaw--Curtis sparse grids. In the schemes presented, the nodes can be selected in a comparatively flexible manner. Three numerical experiments demonstrated that the approach is sound and can cope both with a very large number of nodes and high-dimensional domains.

Even with the tremendous computational simplifications provided by the fully symmetric sets, kernel quadrature rules remain computationally more demanding than most classical quadrature rules. In the end, the decision on which method to use is highly dependent on the computational complexity of evaluating the integrand. Extremes where the rules we have developed are not necessarily useful are easy to identify: (i) in \Cref{sec:exapriori} it is clearly absurd that an integrand as cheap to evaluate as the kernel is evaluated 15 million times while 12 billion kernel evaluations are used to compute the weights, (ii) whereas when the integrand, being for example a complex computer simulation, is very expensive the computational overhead from non-symmetric and likely more accurate kernel quadrature rules is going to be negligible. Consequently, we believe that the method presented in this article is best suited for ``moderately'' expensive integrands in the case when high accuracy is required or probabilistic modelling of uncertainty in the integral estimate desired. This is of course somewhat ambiguous. Precise (and useful) analysis is complicated by, among others, the facts that we do not know how the accuracy of a fully symmetric kernel quadrature rule compares to that of a non-symmetric one (e.g.\ \Cref{thm:RGconv} is only about rates, not the associated constant coefficients) and that it is difficult to account for the value---nor is it easy to decide how much one should value this measure to begin with---one places on the uncertainty measure.

Besides what was discussed in the preceding paragraph, there is a number of topics that could be pursued in the future. These include
\begin{itemize}
\item Developing principled methods for choosing the kernel length-scale for large-scale problems.
\item Proper probabilistic approach to large-scale integration problems. We anticipate that much can be gained in pursuing this direction.
\item As discussed in \Cref{sec:smolyak,sec:wcemin}, there is much room for improvement via optimisation of the fully symmetric sets.
\item Rows of the submatrices $\mK_{ij}$ in \Cref{eq:blockmatrix} typically contain several non-distinct elements. Minor computational improvements might be possible.
\end{itemize}

\section*{Acknowledgements}

We thank Fran\c{c}ois-Xavier Briol, Jon Cockayne, Chris Oates, Jens Oettershagen, and Filip Tronarp for discussion and constructive comments. Suggestions by the anonymous reviewers helped to improve many parts of the article.

\bibliographystyle{siamplain}
\bibliography{references}

\end{document}